\documentclass[12pt]{article}   	% use "amsart" instead of "article" for AMSLaTeX format
\usepackage{geometry}                		% See geometry.pdf to learn the layout options. There are lots.
\usepackage[parfill]{parskip}    		% Activate to begin paragraphs with an empty line rather than an indent
\usepackage{graphicx}				% Use pdf, png, jpg, or eps§ with pdflatex; use eps in DVI mode
\usepackage[pdftex,pagebackref,letterpaper=true,colorlinks=true,pdfpagemode=none,urlcolor=blue,linkcolor=blue,citecolor=blue,pdfstartview=FitH]{hyperref}
\usepackage{amsmath,amsfonts}
\usepackage{color}
\usepackage{amsthm}
\usepackage{boolexpr}
\usepackage{amssymb,latexsym,pstricks,pst-plot}
\usepackage[english]{babel}
\usepackage{pgf}
\usepackage{tikz}
\usetikzlibrary{arrows,automata}
\usepackage[latin1]{inputenc}
\usepackage[UKenglish]{isodate}% http://ctan.org/pkg/isodate
\allowdisplaybreaks

\usepackage{amssymb}
\newtheorem{theorem}{Theorem}[section]
\newtheorem{lemma}[theorem]{Lemma}

\newtheorem{corollary}[theorem]{Corollary}
\newtheorem{observation}[theorem]{Observation}

\title{The Catalan numbers have no forbidden residue modulo primes}
\author{Rob Burns}
\begin{document}
\maketitle
\begin{abstract}
Let $C_n$ be the $n$th Catalan number. For any prime $p \geq 5$ we show that the set $\{C_n : n \in \mathbb{N} \}$ contains all residues mod $p$. In addition all residues are attained infinitely often. Any positive integer can be expressed as the product of central binomial coefficients modulo $p$. The directed sub-graph of the automata for $C_n \mod p$ consisting of the constant states and transitions between them has a cycle which visits all vertices.
\end{abstract}

\section{Introduction}
%\section{}
%\subsection{}
The {\it Catalan numbers\/} are defined by
$$
C_n := \frac{1}{n+1}\binom{2n}{n}.
$$
This note is an addendum to our paper \cite{1701.02975}. In that paper we analysed the Catalan numbers modulo primes $\geq 5$ using automata. Refer to that paper and to \cite{RY2013} for details of how automata can be used to study Catalan numbers and other sequences. 

A set $S$ is said to have a forbidden residue $r$ modulo $p$ if no element of $S$ is $\; \equiv r \mod p$. We show below that the Catalan numbers have no forbidden residue modulo any prime. Garaeva, Luca and Shparlinski \cite{GLS2006} established this result for sufficiently large primes. They also showed that in a certain sense the distribution of $C_n \mod p$ amongst the non-zero residue classes is roughly equal. They also proved that the set \{$C_n: n \leq p^{13/2}(\log p)^6 \, \}$ already includes all residue classes modulo $p$. Our results do not say anything about how quickly $C_n$ covers all residue classes or about how often proportionally each residue class is attained. We do show that each residue class is attained infinitely often. The result for $C_n \mod p$ differs from the situation for powers of primes. Eu, Liu and Yeh \cite{Eu2008} showed that $3$ is a forbidden residue for $C_n$ modulo $4$ and $\{ \, 3, 7 \, \}$ are forbidden residues for $C_n \mod 8$. Liu and Yeh in \cite{Liu2010} calculated $C_n \mod 16$ and $\mod 64$ and thereby determined the forbidden residues in each case. They also showed that $C_n$ has forbidden residues $\mod 2^k$ for any $k$. Forbidden residues for $C_n$ modulo $\{ 32, 64, 128, 256, 512 \}$ were calculated by Rowland and Yassawi using automata in \cite{RY2013}. Kauers, Krattenthaler and M\"uller calculated the generating function for $C_n \mod 4096$ in terms of a special function and so, in theory,  could determine forbidden residues in this case. Similarly Krattenthaler and M\"uller \cite{KM2013} determined the generating function for $C_n \mod 27$ in terms of a special function.

Another difference between the $\mod p$ case and when higher powers of $p$ are involved is that in the $\mod p$ situation all residues are attained infinitely often. Rowland and Yassawi showed that some residues for $C_n \mod \{8, 16 \}$ are attained only finitely many times.

\bigskip

\section{Results}

Let $p$ be prime and let the set $S$ be defined as the multiplicative closure in $\biggr(\, \frac{\mathbb{Z}}{p \mathbb{Z}} \, \biggr)^{\times}$ of the set of elements
$$
\{ \, \binom{2d}{d} \mod p :  \, 0 \leq d \leq \frac{p-1}{2} \}. 
$$
All elements of $S$ are non-zero as $2d \leq p-1$ for $d \leq \frac{p-1}{2}$. We showed in \cite{1701.02975} that $S$ is contained in the set of constant states of the automaton for $C_n \mod p$ and hence that the elements of $S$ appear as residues of $ C_n \mod p$ for some $n$. In explanation of this remark, as shown in \cite{1701.02975}  we have for $d \leq \frac{p-1}{2}$
$$
\Lambda_{d, d} (\, 1*Q^{p-1} \,) = \binom{2d}{d}.
$$
Then if $c_1 = \binom{2d_1}{d_1}$ and $c_2 = \binom{2d_2}{d_2}$ we have
\bigskip
$$
c_1 c_ 2 = \Lambda_{d_1, d_1} \biggr( \; \Lambda_{d_2, d_2}  (\, 1*Q^{p-1} \,) \; \biggr).
$$
Therefore $c_1 c_2$ is also a constant state of the automaton for $C_n \mod p$ and therefore also a residue of $C_n \mod p$ for some $n$.

\bigskip 

\begin{lemma}
\label{lemma1}
The set $S$ contains all non-zero residues modulo $p$.
\end{lemma}
\begin{proof}
Since $S$ is multiplicatively closed it is enough to show that $S$ contains all primes $q < p-1$. We observe that if $c \in S$ then $c^{-1} \mod p = c^{p-2} \mod p$ is also in $S$. We proceed by induction on the set of primes. Firstly, $1 = \binom{0}{0}$ and $2 = \binom{2}{1} \in S$. Let $q < p-1$ be prime. Then $\frac{q+1}{2} \leq \frac{p-1}{2}$ and
\bigskip
$$
\binom{q+1}{\frac{q+1}{2}} \in S
$$
with 
$$
\binom{q+1}{\frac{q+1}{2}} = q r
$$
\bigskip

where $r$ is the product of primes strictly less than $q$. Then by induction $r \in S$ and by the observation above $r^{-1} \in S$. So
$$
q =  \binom{q+1}{\frac{q+1}{2}} r^{-1} \in S.
$$
\end{proof}
\bigskip
\begin{corollary}
For any prime $p \geq 5$, the Catalan numbers have no forbidden residue modulo $p$.
\end{corollary}
\begin{proof}
Lemma \ref{lemma1} shows that all non-zero residues appear in $C_n \mod p$. In addition, the values of $\; n: C_n \equiv 0 \mod p \;$ are plentiful, having asymptotic density $1$ (see \cite{1701.02975}).
\end{proof}

\bigskip
\begin{corollary}
\label{decomp}
For $p$ prime any $n \in \mathbb{N}$ can be written as
$$
n = \prod_i \binom{2d_i}{d_i} \mod p
$$
for suitable choices of $\{ d_i \in \mathbb{N} \}$.
\end{corollary}
\begin{proof}
The same inductive argument as in Lemma \ref{lemma1} can be used to prove the corollary. The choice of $\{ d_i \}$ is not necessarily unique.
\end{proof}  

\bigskip

The set of states and transitions for the automata of $C_n \mod p$ is a directed graph with the states as vertices and transitions as directed edges. The sub-graph $G$ consisting of the non-zero constant states and transitions between them is also a directed graph.
\bigskip

\begin{corollary}
\label{trans}
The directed graph $G$ formed by the non-zero constant states and transitions has a cycle which visits all vertices in $G$.
\end{corollary}
\begin{proof}
Let $c_1$ and $c_2$ be two constant states. It is enough to show that there is a directed state path from $c_1$ to $c_2$. Firstly, $c_2 c_1^{-1}$ is also a constant state by the multiplicative closure of the set $S$. From corollary~\ref{decomp} there are $\{ d_i \}$ such that
$$
c_2 c_1^{-1} = \prod_i \binom{2d_i}{d_i} \mod p.
$$
Then since for constants $c$
$$
\Lambda_{d, d} (\, c*Q^{p-1} \,) = c \Lambda_{d, d} (\, 1*Q^{p-1} \,) = c \binom{2d}{d},
$$
we have modulo $p$
\begin{equation*}
\label{prod}
\biggr(\, \prod_i  \Lambda_{d_i, d_i} \biggr)\,  (\, c_1*Q^{p-1} \,) = c_1 \prod_i \binom{2d_i}{d_i} = c_2 \mod p.
\end{equation*}

Since the application of each $\Lambda_{d_i, d_i}$ corresponds to a transition between states, the product of the $\Lambda_{d_i, d_i}$ corresponds to a directed path from $c_1$ to $c_2$.
\end{proof}  

\bigskip
\begin{observation}
Each residue is attained infinitely often. Firstly from \cite{1701.02975} numbers $n$ which have base $p$ representations containing only digits from the set $\{ 0, 1, ... , \frac{p-1}{2} \}$ have a state path which ends at a non-zero constant state. Since $C_n \mod p$ is the value of the end state of the state path for $n$, it is non-zero $\mod p$. So at least one non-zero constant state (and so at least one non-zero residue) is attained infinitely often. Secondly, the existence of a cycle in the directed graph of the constant states shows that all non-zero constant states are visited infinitely often.
\end{observation}

\bibliographystyle{plain}
\begin{small}
\bibliography{ref}

\begin{thebibliography}{1}

\bibitem{1701.02975}
Rob Burns.
\newblock {S}tructure and asymptotics for {C}atalan numbers modulo primes using
  automata.
\newblock 2017.

\bibitem{Eu2008}
Sen-Peng Eu, Shu-Chung Liu, and Yeong-Nan Yeh.
\newblock Catalan and {M}otzkin numbers modulo 4 and 8.
\newblock {\em European Journal of Combinatorics}, 29:1449--1466, 2008.

\bibitem{GLS2006}
Moubariz~Z. Garaeva, Florian Luca, and Igor~E. Shparlinski.
\newblock Catalan and {A}p{\'e}ry numbers in residue classes.
\newblock {\em Journal of Combinatorial Theory Series A}, 113(5):851 -- 865,
  July 2006.

\bibitem{KM2013}
Christian Krattenthaler and Thomas~W. M\"uller.
\newblock A method for determining the mod-$3^k$ behaviour of recursive
  sequences.
\newblock {\em ArXiv}, arXiv:1308.2856:82, 2013.

\bibitem{Liu2010}
S.-C. Liu and J.~Yeh.
\newblock Catalan numbers modulo $2^k$.
\newblock {\em J Integer Sequences}, 13, 2010.

\bibitem{RY2013}
Eric Rowland and Reem Yassawi.
\newblock Automatic congruences for diagonals of rational functions.
\newblock {\em ArXiv}, arXiv:1310.8635:42, 2013.

\end{thebibliography}
\end{small}

\end{document}